\documentclass[11pt,a4paper,psamsfonts]{amsart}
\usepackage{amssymb,amscd,amsxtra,calc}
\usepackage{cmmib57}
\usepackage{graphicx}

\usepackage{lipsum}
\usepackage{amsfonts}
\usepackage{graphicx}
\usepackage{epstopdf}
\usepackage{algorithmic}
\usepackage{float}
\restylefloat{table}
\usepackage{placeins}
\usepackage{array}
\ifpdf
  \DeclareGraphicsExtensions{.eps,.pdf,.png,.jpg}
\else
  \DeclareGraphicsExtensions{.eps}
\fi

\usepackage{subcaption}

\theoremstyle{plain}
    \newtheorem{thm}{Theorem}[section]

    \newtheorem{lemma}[thm]{Lemma}
    \newtheorem{proposition}[thm]{Proposition}

    \newtheorem{theorem}[thm]{Theorem}

\theoremstyle{definition}

\theoremstyle{remark}
    
    \newtheorem{example}[thm]{Example}

\newcommand{\authorfootnotes}{\renewcommand\thefootnote{\@fnsymbol\c@footnote}}

\usepackage{caption} 
 \title[Unconstrained optimisation on Riemannian manifolds]{Unconstrained optimisation on Riemannian manifolds}
 \author{Tuyen Trung Truong}
   \address{Department of Mathematics, University of Oslo, Blindern 0851 Oslo, Norway}
  \email{tuyentt@math.uio.no}
\thanks{Department of Mathematics, University of Oslo, Blindern 0851 Oslo, Norway, Email address: tuyentt@math.uio.no}

    \date{\today}
   
   \subjclass[2010]{}

\begin{document}
\maketitle

\begin{abstract}
In this paper, we give explicit descriptions of versions of (Local-) Backtracking Gradient Descent and New Q-Newton's method to the Riemannian setting. Then, we combine Nash's embedding theorem and some good local retractions to extend our previous convergence results for the above methods on Euclidean and Banach/Hilbert spaces to Riemannian manifolds. Here are some easy to state  consequences of results in this paper, where X is a general Riemannian manifold of finite dimension and $f:X\rightarrow \mathbb{R}$ a $C^2$ function which is Morse (that is, all its critical points are non-degenerate). 

{\bf Theorem.} For random choices of the hyperparameters in the Riemanian Local Backtracking Gradient Descent algorithm and for random choices of the initial point $x_0$, the sequence $\{x_n\}$ constructed by the algorithm either (i) converges to a local minimum of $f$ or (ii) eventually leaves every compact subsets of $X$ (in other words, diverges to infinity on $X$). If $f$ has compact sublevels, then only the former alternative happens. The convergence rate is the same as in the classical paper by Armijo. 

{\bf Theorem.} Assume that $f$ is $C^3$. For random choices of the hyperparametes in the Riemannian New Q-Newton's method, if the sequence constructed by the algorithm converges, then the limit is a critical point of $f$. We have a local Stable-Center manifold theorem, near saddle points of $f$, for the dynamical system associated to the algorithm. If the limit point is a non-degenerate minimum point, then the rate of convergence is quadratic. If moreover $X$ is an open subset of a Lie group and the initial point $x_0$ is chosen randomly, then we can globally avoid saddle points.  

We provide several experiments showing that treating a Euclidean optimisation, with a singular cost function or with constraints, as a Riemannian optimisation  could be beneficial. As an application, we propose a  general method using Riemannian Backtracking GD to find minimum of a function on a bounded ball in a Euclidean space, and do explicit calculations for calculating the smallest eigenvalue of a symmetric square matrix. 

\end{abstract}


\subsection{Introduction}

Optimisation is important in both academic research and real life applications, ranging from mathematics, physics, computer science, to industrial production and commercial products. One exciting recent development is the paradigm of Deep Learning, where a challenging task (such as playing games, image and video classifications, document and sound processing - one can ready see applications of these in one's smart phones) can be deduced to a large scale optimisation problem. While the major impression for Deep Learning comes from industry, there are also directions in theoretical mathematics utilising it. For example, an active research direction in Automated proof checking is to utilise achievements in Natural Language Processing to translate informal written mathematical proofs to formal proofs, which then can be checked by computers for the correctness. This can be handy in difficult situations, for example when the proof is too large or too specialised to follow. Besides the benefits to mathematics itself, the success of this program in turn can have potentially further applications in computer science and hence the society at large. Indeed, roughly speaking, the Curry-Howard correspondence says that checking the correctness of a computer program is the same as checking the correctness of a corresponding mathematical proof.   

Note that, since finding global minima is NP-hard and also because of other reasons such as the compute expense and deadlines, practically one must rely on iterative algorithms to solve large scale optimisation. As such, it is important to have a good iterative algorithm which has strong theoretical guarantee under conditions which are both general and easy to check on the code functions. 

The most familiar setting of optimisation is on Euclidean spaces. For example, training Deep Neural Networks can be deduced to an optimisation on the Euclidean space of parameters (i.e. weights and biases) of the Deep Neural Network. One considers a function $f:\mathbb{R}^m\rightarrow \mathbb{R}$, and wants to find a local minimum. A local minimum is a critical point, that is, it satisfies a first order condition $\nabla f(x_{\infty})=0$. To be able to discuss more, we assume (in this paragraph only, except otherwise stated) that $f$ is moreover $C^2$ near $x_{\infty}$. We say that $x_{\infty}$ is non-degenerate if $\nabla ^2f(x_{\infty})$ is invertible. We say that $x_{\infty}$ is a saddle point if $\nabla ^2f(x_{\infty})$ is invertible and has both positive and negative eigenvalues. We say that $x_{\infty}$ is a generalised saddle point if $\nabla ^2f(x_{\infty})$ has at least one negative eigenvalue. Since saddle points are dominant for functions in higher dimensions and can have bad behaviour \cite{bray-dean, dauphin-pascanu-gulcehre-cho-ganguli-bengjo},  both for general functions and for functions encountered in Deep Learning, it is important to guarantee that the limit point ({\bf if exists}) is not a (generalised) saddle point. 

In Euclidean optimisation, one usually uses first order optimisation algorithms (represented by  gradient descent methods) and second order optimisation algorithms (represented by Newton's methods). Below is a very brief overview of these, please see the references and references therein for more detail. 

The general scheme for gradient descent (GD) method \cite{cauchy} is as follows: We start from an initial point $x_0$, and construct a sequence $x_{n+1}=x_n-\delta _n\nabla f(x_n)$, where $\delta _n>0$ is an appropriately chosen positive number. The most basic form of these methods is the Standard GD, where $\delta _n=\delta _0$ is a constant. Another basic form (Backtracking GD) is to require Armijo's condition: One fixes a positive number $0<\alpha <0$ and checks that 
\begin{eqnarray*}
f(x_n-\delta _n\nabla f(x_n))-f(x_n)\leq -\alpha \delta _n||\nabla f(x_n)||^2. 
\end{eqnarray*}
The main difference is that in Backtracking GD one uses function evaluations, which on the one hand  can be expensive, but on the other hand guarantee the descent property. There are over $100$ modifications of Standard GD, and a few modifications of Backtracking GD. Some of the most popular modifications of Standard GD are NAG, Momentum, Adam, Adadelta, and RMSProp (see \cite{ruder}), while some modifications of Backtracking GD are Local Backtracking GD and Continuous form of Backtracking GD \cite{truong} and Unbounded Backtracking GD \cite{truong4}. Most of the existing literature uses modifications of Standard GD, and can prove theoretical results only under extra assumptions such as $f$ is $C^{1,1}_L$ (meaning that the gradient $\nabla f$ is $C^{1,1}_L$, and where the learning rates $\delta _n$ are required to be of the order $1/L$)  and/or convexity of $f$, and the typical proven results are: i) If $\{x_{n_j}\}$ is a convergent subsequence, then $\lim _{j\rightarrow\infty}\nabla f(x_{n_j})=0$ (in many papers, this property is called "convergence", which could make readers confused with the rigorous mathematical notion of "convergence" for the sequence $\{x_n\}$ itself), and ii) for a random choice of the initial point $x_0$, if the sequence $\{x_n\}$ converges to $x_{\infty}$, then $x_{\infty}$ cannot be a generalised saddle point \cite{lee-simchowitz-jordan-recht, panageas-piliouras}. 

We note that usually a result on GD \cite{robbins-monro}, where among many things it is required that $\lim _{n\rightarrow\infty}\delta _n=0$, seems a standard reference to justify the use of GD in the stochastic setting (such as with mini-batches in Deep Neural Networks). However,  a very recent implementation of it \cite{merti-etal} shows that the performance (archiving about $84\%$ test accuracy for the dataset CIFAR10 \cite{cifar} on the Deep Neural Network Resnet18 \cite{He}) is not very good compared with the popular modifications of Standard GD or of Backtracking GD. Why its performance is not extremely good can be explained as follows (see \cite{truong-nguyen} for more detail): we expect that the sequence $\{x_n\}$ converges, in which case $\delta _n$ should vary very little in the long run (close to the best learning rate one can choose at the limit point) and hence in particular we should not choose $\lim _{n\rightarrow\infty}\delta _n=0$ because this can slowdown convergence and/or force convergence to bad points.  

To our knowledge, while not widely used at the moment, Backtracking GD and its modifications currently are best theoretically guaranteed. Here are some properties: 

- The sequence $\{x_n\}$ either converges to a critical point of $f$ or diverges to infinity, {\bf if}: $f$ is real analytic or more generally satisfies the so-called Losjasiewicz gradient condition \cite{absil-mahony-andrews} or f has at most countably many critical points \cite{truong-nguyen}. In the general case, one can draw some useful conclusions \cite{truong-nguyen}. 

- A modification, called Local Backtracking GD allows avoidance of saddle points, while still has the convergence result in the previous paragraph \cite{truong}. 
Another modification, which allows the choice of learning rates to depend {\bf continuously} on the points $x_n$, also has the same properties as described in the previous sentence \cite{truong}. 

- Some further modifications (including combining with Momentum and NAG), with good theoretical justifications, allow to implement Backtracking GD into large scale optimisation in Deep Neural Networks \cite{truong-nguyen}, and for a more recent work see \cite{vaswani-etal}. The performance (on various datasets, including CIFAR10 and CIFAR100 \cite{cifar}) is very good and stable (when the hyperparameters are changed). For example, on CIFAR10 and Resnet18, \cite{truong-nguyen} reports that the validation accuracy for Backtracking GD is 91.64$\%$, while for a combination of Backtracking GD and Momentum achieves 93.70$\%$ and yet another combination of Backtracking GD and NAG achieves 93.85$\%$. (Similar performance is reported subsequently by other authors \cite{vaswani-etal}.) These validation accuracies are stable across different choices of initial learning rates, and the time needed to run them is not much slower than that for running Standard GD (which has no systematical and good way to pick out a learning rate to start with, except using either grid search or some heuristics lacking good theoretical justification). 

- The Unbounded Backtracking GD version \cite{truong4}  -  in which learning rates are not required to be uniformly bounded - can indeed improve the performance, compared to the original Backtracking GD, if the sequence $\{x_n\}$ converges to a {\bf degenerate} critical point \cite{truong-nguyen2}. 

- While this fact is trivial, it is worth noting that if $f$ is in $C^{1,1}_L$ and the learning rate $\delta$ in Standard GD is chosen so that $\delta <1/L$, then the Standard GD schedule is a special case of Backtracking GD. 

One important point which makes Backtracking GD different is that from Armijo's condition one can prove that $\lim _{n\rightarrow\infty}||x_{n+1}-x_n||=0$. A simple but crucial fact which has been used in \cite{truong-nguyen} and in our other papers is that $\mathbb{R}^m$ is, {\bf topologically}, a subspace of the real projective space $\mathbb{P}^m$, while the metric on $\mathbb{R}^m$ is bounded from below by the spherical metric on $\mathbb{P}^m$. This allows us to use the results in \cite{asic-adamovic} for {\bf all} subsets of $\mathbb{R}^m$ and not just to compact subsets of $\mathbb{R}^m$ as in prior work (such as Chapter 12 in \cite{lange}). 

If one wants to have faster convergence (i.e. less iterations), then one can utilise second order methods. A prototype of these is Newton's method, which applies to a $C^2$ function $f$. We start from an initial point $x_0$, and proceed as follows: if the Hessian $\nabla ^2f(x_n)$ is invertible, then we define $x_{n+1}=x_n-\nabla ^2f(x_n)^{-1}.\nabla f(x_n)$.  The good thing about Newton's method is that if it converges, then usually it converges very quick (rate of convergence can be quadratic). An extreme case is for $f$ a quadratic function, when Newton's method converges after $1$ step, no matter what type the function $f$ is. However, compared to (Backtracking) GD, Newton's method has the following drawbacks. First, it does not guarantee convergence. Second, it has the tendency to be attracted to the nearest critical point, even if that is a saddle point or local maximum. Third, it is expensive to run Newton's method, which is an enormous obstacle in large scale optimisation (for example, in current popular Deep Neural Networks, one needs to work with spaces of dimensions in the size of $10^9$). Fourth, there is an uncertainty of what to do when the Hessian $\nabla ^2f(x_n)$ is not invertible. 

There are so-called quasi-Newton's methods, which aim to resolve the third point above by using rank 1 matrices, for example Gauss-Newton, BFGS and SR-1, see e.g. \cite{absil-etal, boumal} . These also try to overcome the fourth point above by adding a very big matrix to make the resulting a positive matrix. However, for all these methods, there are no theoretical guarantees about convergence and avoidance of saddle points. Also, it is not clear if these methods have the same convergence rate as the original Newton's method. 

In a recent paper \cite{truong-etal}, we proposed a new modification of Newton's method called New Q-Newton's method, which has the following properties: i) it provides a simple way to deal with the case the Hessian is not invertible (by adding a term of the form $\delta ||\nabla f(x)||^{\alpha}$ into $\nabla ^2f(x)$ whenever the latter is not invertible), ii) it avoids saddle points, and iii) if it converges then the limit point is a critical point of $f$ and the rate of convergence is the same as that of the original Newton's method. Experiments show that it performs very well in small scale optimisation, against various other modifications of Newton's method. The only remaining drawbacks  are that it does not guarantee convergence and it is expensive to use in large scale optimisation.  

There are, among many other, two ways one can go to  generalise the above results: one is to go to infinite dimension (that is Hilbert and Banach spaces) and another is to go to Riemannian manifolds.    

The infinite dimension setting is natural for example for solving PDE, in particular those coming from physics. Again, in this case, most of the existing work uses modifications of Standard GD, and requires assumptions such as $C^{1,1}_L$ and/or convexity,  please check the references in \cite{truong3} and references therein. We note that in this setting, also (Local-) Backtracking GD can be defined and results similar to the finite Euclidean space case can be established \cite{truong3}. For the possibility of extending New Q-Newton's method to this setting, please see the discussions in \cite{truong-etal}. 

The Riemannian manifold setting also offers many interesting and useful applications, such as in the Netflix prize competition \cite{netflix}. Another example is Fisher information metric, which is a natural choice of Riemannian metric on a space of parametrised statistical models (such as Deep Neural Networks), the modern theory is largely due to S. Amari \cite{amari-nagaoka}. Yet one other interesting case is that of constrained optimisation problems on Riemannian manifolds, if the constraints give rise naturally to a Riemannian submanifold. Again, in this case, most of the existing work uses modifications of Standard GD or quasi-Newton's methods, and obtains similar results and has similar drawbacks as described before. Besides  global constraints on the cost function $f$ as mentioned before ($C^{1,1}_L$ and or convexity), usually new global geometric constraints, on bounds for curvatures of the Riemannian manifold, are also required. A standard textbook for this subject is \cite{absil-etal}, and a more up-to-date reference is \cite{boumal}. We note that the cases considered so far in most of the existing literature are for matrix manifolds.

The main purpose of this paper is to define some explicit versions of (Local) Backtracking GD and New Q-Newton's method, and extend the mentioned properties to the Riemannian manifold setting. In Example in \cite{truong2}, we illustrated that for Backtracking GD in the case the Riemannian manifold is an open subset $\Omega$ of a Euclidean space: in which case the learning rate must be $<dist(x,\partial \Omega )/||\nabla f(x)||$, and we can understand the statement "$\{x_n\}$ diverges to infinity" as that any cluster point of $\{x_n\}$ is on the boundary of the open set. In some previous work, we alluded to such extensions for (Local) Backtracking GD, mentioning that Backtracking GD is local in nature, but without further details and under a technical assumption that there is a compact metric space $Z$ and an embedding $X$ into $Z$ so that the Riemannian metric on $X$ is bounded from below by the induced metric from $Z$. An open question there was to extend the results to general Riemannian manifolds. In this paper, we show that the extension is available on all Riemannian manifolds by using Nash's embedding theorem \cite{nash1, kuiper, nash2} and the fact mentioned above concerning $\mathbb{R}^m$ and $\mathbb{P}^m$. The details, together with the New Q-Newton's method on Riemannian manifolds and some other additions, are written down explicitly in this paper. 

We remark that in this paper we use only {\bf local} quantities in the definitions and results, in particular relying only on {\bf local} Retractions (existing on all Riemannian manifolds) - see Subsection \ref{SubsectionPreliminaries}. Local quantities allow flexibility and ease of use. For example, with a glance one could say quickly whether a given function is $C^1$ or $C^2$, or locally Lipschitz continuous. (In contrast, global properties such as Lipschitz continuous or convexity can be cumbersome to define - need to use global parallel transport - and difficult to check or be satisfied in practice.) Also, in many cases, we have no restrictions on the geometry (such as bounds on curvatures) of the  Riemannian manifold in question.  Here are some easy-to-state consequences of results in this paper. 

\begin{theorem} Let $X$ be a Riemannian manifold of finite dimension, and $f:X\rightarrow \mathbb{R}$ a $C^2$ function which is Morse (that is, all its critical points are non-degenerate). Then for random choices of the hyperparameters in the Local Backtracking Gradient Descent algorithm (which depends only on {\bf local} Lipschitz constants of the gradient of $f$) and for random choices of the initial point $x_0$, the sequence $\{x_n\}$ constructed by the algorithm either (i) converges to a local minimum of $f$ or (ii) eventually leaves every compact subsets of $X$ (in other words, diverges to infinity on $X$). If $f$ has compact sublevels, then only the former alternative happens. The convergence rate is the same as in the classical paper by Armijo. 
\label{TheoremConsequence1}\end{theorem}

\begin{theorem} Let $X$ be a Riemannian manifold of finite dimension, and $f:X\rightarrow \mathbb{R}$ a $C^3$ function which is Morse (that is, all its critical points are non-degenerate). For random choices of the hyperparametes in the New Q-Newton's method, if the sequence constructed by the algorithm converges, then the limit is a critical point. We have a Stable-Center manifold theorem, near saddle points of $f$, for the dynamical system associated to the algorithm. If the limit point is a non-degenerate minimum point, then the rate of convergence is quadratic. If moreover $X$ is an open subset of a Lie group and the initial point $x_0$ is chosen randomly, then we can globally avoid saddle points.  
\label{TheoremConsequence2}\end{theorem}

The above  two theorems are special cases of Theorems \ref{Theorem2} and \ref{Theorem4}, respectively. The remaining of this paper is organised as follows. In the next subsection, we present very briefly some backgrounds in Riemannian geometry necessary for later use. In Subsection 0.3, we present the new algorithms. In relation to avoidance of saddle points for the Riemannian version of New Q-Newton's method, we introduce a property of (local) retractions called Real analyic-like. The latter property is satisfied for example by all open subsets of real Lie groups.  In the subsections after that, we present consecutively main results and proofs, as well as several experiments showing the advantage of using the Riemannian algorithms developed in this paper  even when one works on a Euclidean space - but with a singular cost function or with constraints. As an application, we propose a  general method using Riemannian Backtracking GD to find minimum of a function on a bounded ball in a Euclidean space, and do explicit calculations for calculating the smallest eigenvalue of a symmetric square matrix. We finish the paper with  some conclusions and ideas for future work. 

{\bf Remarks.} A heuristic argument for why Backtracking GD and New Q-Newton's method have good theoretical guarantee (as well as good practical performance) is that they have "correct units" (see \cite{truong-nguyen2}). There it was observed also that the Diminishing learning rate method does not have "correct units", while Standard GD for functions $f\in C^{1,1}_L$ with learning rates in the order of $1/L$ does have "correct units" (as mentioned above, in this case Standard GD is a special case of Backtracking GD).  The notion of "correct units" was defined by Zeiler in his Adadelta's paper \cite{Zeiler}, where he showed that Newton's method and Adadelta have "correct units", while Standard GD, Adam and PropRMS do not have "correct units". His recommendation was that one should use only methods which have "correct unit", and if a method does not have "correct unit" then one should modify it to another which has "correct units". From experimental results so far, it seems to us that Zeiler's recommendation does have a seed of truth in it.  

{\bf Remarks.} The definitions and results can be easily extended to modifications such as Unbounded Backtracking GD. However, to keep the paper succinct, in the remaining of the paper, we will treat only the original Backtracking GD. We also note that it takes considerably more work to extend New Q-Newton's method than to extend Backtracking GD methods.   

{\bf Acknowledgements.} This work is supported by Young Research Talents grant number 300814 from Research Council of Norway. 

\subsection{Preliminaries on Riemannian geometry}\label{SubsectionPreliminaries} This is just a very terse overview of backgrounds from Riemannian geometry needed for later use. We refer the interested readers to \cite{lee} (for generalities about Riemannian manifolds) and \cite{absil-etal, boumal} (for more details on how to use the tools for optimisation on Riemannian manifolds). 

A Riemannian manifold is a manifold $X$, together with a Riemannian metric $g(x)$ which is an inner product on tangent spaces $T_xX$. Usually we assume that $g(x)$ varies smoothly when $x$ changes. With the help of a Riemannian metric, and the associated Levi-Civita connection, given a function $f:X\rightarrow \mathbb{R}$ we can define the notions of gradient $grad(f)\in TX$ and Hessian $Hess(f)$ which reduce to the familiar notions $\nabla f$ and $\nabla ^2f$ when $X$ is a Euclidean space. 

A Riemannian metric gives rise to a metric $d_X$ on $X$ in the following manner. If $\gamma :[a,b]\rightarrow X$ is a smooth curve, then we define the length of $\gamma$ as:

\begin{eqnarray*}
L(\gamma )=\int _a^b||\gamma '(t)||dt.
\end{eqnarray*}
Here $||\gamma '(t)||$ is the length of the vector $\gamma '(t)$, with respect to the inner product given by the Riemannian metric $g$. 

Given $x,y\in X$, the distance $d_X(x,y)$ is the infimum  of the  lengths $L(\gamma )$, where $\gamma $ runs on all over curves on $X$ connecting $x$ and $y$. 

A geodesic is a curve which realises the distance. It satisfies a second order ODE. By results from ODE, uniqueness and local existence of geodesics are guaranteed. In particular, one can define exponential maps. If $x\in X$ and $v\in T_xX$ with $||v||$ small enough, then there is a unique geodesic $\gamma _v:[0,1]\rightarrow X$ such that $\gamma _v(0)=x$ and $\gamma _v '(0)=v$. The exponential map is $exp _x(v)=\gamma _v(1)$.  

For $x\in X$ and $r>0$, we denote by $B(T_xX,r)$ the set $\{v\in T_x:~||x||<r\}$. The injectivity radius, $inj(x)$, is the supremum of all $r$, for which the exponential map $exp_x$ is well-defined and a diffeomorphism from $B(T_x,r)$ onto its image. By the previous paragraph, we always have $inj (x)>0$. The injectivity radius can be infinity, for example in the case $X$ is a Euclidean space. We note the following important property of injectivity radius (see Proposition 10.18 in \cite{boumal}). 

\begin{proposition} The map $inj:~X\rightarrow (0,\infty ]$ is continuous. 
\label{PropositionInjectivityRadius}\end{proposition}

For a complete Riemannian manifold, a proof can be found in \cite[Proposition 10.37]{lee}. For a general Riemannian manifold, a proof is given by M. Stephen and J. Lee in an online discussion \cite{inj} and is incorporated into  \cite[Section 10.8]{boumal}. 

Exponential maps provide a way to move, on the same manifold, from a point $x$ in a chosen direction $v$. However, in optimisation, the most crucial property of an exponential map is that its derivative at $0\in T_xM$ is the identity map. (This implies, by inverse function theorem, that the exponential map is a local diffeomorphism near $0$.) This is generalised to the following (global) notion of "retraction", as given in \cite[Definition 4.1.1]{absil-etal} which we now recall. 

{\bf Definition (Global retraction).} A retraction on a Riemannian manifold $X$ is  a smooth mapping $R:TX\rightarrow X$ from the tangent bundle $TX$, with the following properties. If $R_x=R|_{T_xX}$ then: i) $R_x(0_x)=x$ where $0_x$ is the zero element of $T_xX$, and ii) $DR_x(0_x)=Id_{T_xX}$.  

On Euclidean spaces or complete Riemannian manifolds, global retractions exist. Some other interesting global retractions are given in \cite[Section 4.1]{absil-etal}. On the other hand, it is not clear if all Riemannian manifolds have at least one retraction in the above sense (we think that probably the answer is No). For example, by Hopf-Rinow theorem, the exponential map is defined on the whole tangent bundle if and only if $X$ is a complete metric space. In the current literature, all theoretical results are stated and proven under the existence of such global retractions. There are also a notion of "local retractions" \cite[Section 4.1.3 ]{absil-etal}, but we are not aware of any use of them in theoretical treatments, the reasons may be that they are not strong enough. In the next subsection, we will discuss a stronger version called  "strong local retractions", which exist on all Riemannian manifolds and which are strong enough to guarantee good theoretical properties.    

Next we discuss some estimates for Taylor's expansion of functions of the form $f(R_x(v))$ where $R_x:~B(T_xX,r(x))\rightarrow X$ a diffeomorphism onto its image such that $R_x(0)=x$ and $DR_x(0)=Id_{T_xX}$. The presentation here is taken from \cite[Section 10.4]{boumal}. We let $\widehat{f_x}=f\circ R_x:B(T_xX,r(x))\rightarrow \mathbb{R}$.  Suppose that 
\begin{equation}
||\nabla \widehat{f_x}(v)-\nabla \widehat{f_x}(w) ||\leq L||v-w||,
\label{Equation1}\end{equation}
for all $v,w\in B(T_xX,r)$, here $L>0$ is a positive constant. Then, 
\begin{equation}
|f(R_x(v))-f(x)-<v,grad f(x)>|\leq L||v||^2/2,
\label{Equation2}\end{equation}
 for all $v\in B(T_xX,r)$. This inequality is the generalisation of the usual property of functions on an open subset of a Euclidean space whose gradient is (locally) Lipschitz continuous. While Equation (\ref{Equation1})  can be complicated for general maps $R_x$ and functions $f$, there is one special case where it has the usual form. Indeed, if $f$ is $C^2$ and $R_x=exp_x$ is the exponential map and $r(x)=inj(x)$, then Equation (\ref{Equation1}) is satisfied for all $x\in X$ if and only if $||Hess (f)||\leq L$ on $X$.

A more general inequality is as follows, see \cite[Exercise 10.51]{boumal}. Let notations be as in the previous paragraph. Suppose that 
\begin{equation}
||\nabla \widehat{f_x}(v)-\nabla \widehat{f_x}(0) ||\leq L||v||,
\label{Equation1}\end{equation}
for all $v\in B(T_xX,r)$, here $L>0$ is a positive constant. Then, 
\begin{equation}
|f(R_x(v))-f(x)-<v,grad f(x)>|\leq L||v||^2/2,
\label{Equation2}\end{equation}
 for all $v\in B(T_xX,r)$. 
 
 The above inequalities can be extended globally by the following trick. If $\gamma :[0,1]\rightarrow X$ is a continuous curve, with $x=\gamma (0)$ and $y=\gamma (1)$, then we can find numbers $t_0=0<t_1<\ldots <t_{N-1}<t_N=1$, so that for all $j$, the point $x_j=\gamma (t_j)$ is in the range of the exponential map $exp_{x_{j-1}}$. Moreover $x_{j}=exp_{x_{j-1}}(v_{j-1})$ where $\sum ||v_j||\sim $ the length of $\gamma$. Hence, with the above local inequalities, one can prove global inequalities like $|f(x)-f(y)|\leq Ld_X(x,y)$, by using a telescope sum $\sum |f(x_j)-f(x_{j-1})|$. 

A diffeomorphism $\iota :X\rightarrow \iota (X)\subset \mathbb{R}^N$ is an isometric embedding if the lengths of vectors are preserved. This leads to preservation of many properties, including distances. We end this with Nash's embedding theorem \cite{kuiper, nash1, nash2}. 

\begin{theorem} Let $X$ be a Riemannian manifold of finite dimension. Then there is an isometric embedding $\iota: X\rightarrow \mathbb{R}^N $ for some big enough dimension $N$. 
\label{TheoremNashEmbedding}\end{theorem}

\subsection{(Local-) Backtracking GD and New Q-Newton's method on Riemannian manifolds}\label{SubsectionAlgorithms} The bulk of this subsection is to extend the (Local-) Backtracking GD and New Q-Newton's method  to a general Riemannian manifold, in such a way that known results for optimisation on Euclidean spaces can be extended. 

In the previous subsection, we mentioned that it is not known if global retractions exist on every Riemannian manifolds. On the other hand, as mentioned, the versions of local retractions in the current literature are not strong enough to guarantee good theoretical properties of iterative optimisations. We will first introduce a version of local retractions, called strong local retraction, which is both existing on all Riemannian manifolds and strong enough to derive good theoretical guarantees. We recall that if $x\in X$ and $r>0$, then $B(T_xX,r)$ $=$ $\{v\in T_xX:~||v||<r\}$. 

{\bf Definition (Strong Local Retraction).}  A strong local retraction consists of a function $r:X\rightarrow (0,\infty ]$ and a map $R:\bigcup _{x\in X}B(T_xX,r(x))\rightarrow X$ with the following properties: 

1) $r$ is upper semicontinuous, that is $\limsup _{y\rightarrow x}r(y)\leq r(x)$.

2) $\bigcup _{x\in X}B(T_xX,r(x))$ is an open subset of $TM$, and $R$ is $C^1$ on its domain. 

3) If $R_x=R|_{B(T_x,r(x))}$, then $R_x$ is a diffeomorphism and $DR_x(0)=Id$. Moreover, we assume that $R$ is $C^1$.

\begin{example}
1) Because of the assumption that $\bigcup _{x\in X}B(T_xX,r(x))$ is an open subset of $TM$, it follows that for every compact set $K\subset X$ we have $\inf _{x\in K}r(x)>0$. 

2) If $R$ is a global retraction, then with the choice of $r(x)=\infty$ for all $x$, we have that $(r,R)$ is a Strong local retraction. 

3) If $r(x)=inj(x)$ and $R=$ the exponential map, then $(r,R)$ is a Strong local retraction, by Proposition \ref{PropositionInjectivityRadius}. Thus on every Riemannian manifold, there exists a Strong local retraction, where moreover we can assume that $r$ is a continuous function.  
\label{Example1}\end{example}

We have the following useful property on local Lipschitz continuity for Strong local retractions. 

\begin{lemma}
Let $r,R$ be a Strong local retraction on a Riemannian manifold $X$.   For each $x\in X$, there exists $s(x),h(x)>0$ such that the following property holds. For all $y,z\in R(B(T_xX,s(x)))$ and $v\in T_yX$, $w\in T_zX$ so that $||v||<s(x)-||R_x^{-1}(y)||$ and $||w||<s(x)-||R_x^{-1}(z)||$, then 
\begin{eqnarray*}
d_X(R_y(v),R_z(w))\geq \frac{1}{2}d_X(y,z)-h(x)||PT_{x\leftarrow y}v-PT_{x\leftarrow z}w||.
\end{eqnarray*}
Here $d_X$ is the induced metric on $X$, and $PT_{x\leftarrow y}$ is the parallel transport of vectors from $T_yX$ to $T_xX$ along the unique geodesic from $x$ to $y$ (when $y$ is close enough to $x$). 
\label{LemmaLowerBoundR}\end{lemma}
\begin{proof}
Indeed, this is a consequence of the assumption that $R$ is $C^1$ (and hence in particular is locally Lipschitz continuous, in both variables $x$ and $v\in B(T_xX,r(x))$) and the assumption that $R_y(0)=y$ for all $y$, which the readers can ready work out on local coordinate charts. 
\end{proof}

Now we are ready to define the versions of (Local-) Backtracking GD and New Q-Newton's method on a general Riemannian manifold.  

{\bf Definition (Riemannian Backtracking GD).} Let $X$ be a Riemannian manifold, and $(r,R)$ a Strong local retraction on $X$.  We choose $0<\alpha , \beta <1$ and $\delta _0>0$. Let $f:X\rightarrow \mathbb{R}$ be a $C^1$ function. For $x\in X$, we choose $\delta (x)$ to be the largest number $\delta $ in the set $\{\beta ^j\delta _0:~j=0,1,2,\ldots \}$ which satisfies both $\delta ||grad (f)(x)||< r(x)/2$ and Armijo's condition: 
\begin{eqnarray*}
f(R_x(-\delta grad (f)(x)))-f(x)\leq -\alpha \delta ||grad (f)(x)||^2. 
\end{eqnarray*}
(Since $R_x(0)=0$ and $DR_x(0)=Id$, there exists at least a positive number $\delta '$ for which Armijo's condition is satisfied. Hence, the function $\delta (x)$ is well-defined.)

The update rule for Riemannian Backtracking GD is as follows: We choose an initial point $x_0$, and construct a sequence $x_{n+1}=R_{x_n}(-\delta (x_n)grad (f)(x_n))$. 

\begin{example}
In the Euclidean space, this definition is classical, goes back at least to \cite{armijo}. In the Riemannian manifold setting, where the retraction $R$ is global, it is is known for awhile \cite{absil-etal, boumal}. In our setting of Strong local retractions here, it is a bit more complicated to state. 
\label{Example0}\end{example}

\begin{lemma}
Let the setting be as in the definition for the Riemannian Backtracking GD algorithm. Let $K\subset X$ be a compact set. If $\inf _{x\in K}||grad (f)(x)||>0$, then $\inf _{x\in K}\delta (x)>0$. 
\label{LemmaLowerBound}\end{lemma}
\begin{proof}
The proof is exactly as in the Euclidean setting, by using observation 1) in Example \ref{Example1}, which one can find for example in \cite{truong-nguyen}. 
\end{proof}

{\bf Definition (Riemannian Local - Backtracking GD)} Let $X$ be a Riemannian manifold, and $(r,R)$ a Strong local retraction on $X$, where it is assumed that $r$ is continuous.  Let $f:X\rightarrow \mathbb{R}$ be $C^1$. Put $\widehat{f_x}=f\circ R_x: B(T_xX,r(x))\rightarrow X$. We assume that there is a continuous function $L:X\rightarrow (0,\infty)$ such that for all $x\in X$ and all $v\in B(T_xX,r(x))$, the following inequality is satisfied: 
\begin{eqnarray*}
||\nabla \widehat{f_x}(v)-\nabla \widehat{f_x}(0)||\leq L(x)||v||. 
\end{eqnarray*}

Moreover, sssume that the conclusions of Lemma \ref{LemmaLowerBoundR} are satisfied with the choice $s(x)=r(x)$ and $h(x)=L(x)$.
 
Fix $0<\alpha ,\beta <1$ and $\delta _0$. For each $x\in X$, we define $\widehat{\delta}(x)$ to be the largest number $\delta $ among $\{\beta ^j\delta _0:~j=0,1,2,\ldots \}$ which satisfies the two conditions: 

\begin{eqnarray*}
\delta &<&\alpha /L(x),\\
\delta ||grad (f)(x)|| &<& r(x)/2.
\end{eqnarray*}

The update of Riemannian Local Backtracking GD is as follows. We choose an initial point $x_0\in X$, and construct the sequence $\{x_n\}$ as follows: 
\begin{eqnarray*}
x_{n+1}=R_{x_n}(-\widehat{\delta}(x_n)grad (f)(x_n)). 
\end{eqnarray*}

\begin{example} This definition was given in  \cite{truong, truong3} in the Euclidean and Hilbert and Banach spaces settings, where it is simpler to state.

(i) If $f$ is in $C^{1,1}_L$ (see \cite{absil-etal, boumal} for precise definition in the Riemannian setting, see also the previous subsection), then we can choose $L(x)=L$ for all $x$. 

(ii) If $f$ is $C^2$ and $R=$ the exponential map, then after  making $r(x)$ to be finite (for example, by replacing it with $\min \{r(x),1\}$), we can choose $L(x)$ to be any continuous function so that $L(x)\geq \sup _{z\in R_x(B(T_xX,r(x)))}||Hess(f)||$ for all $x\in X$. 

iii) More generally, if $f,R$ are $C^2$ functions, then since $\widehat{f_x}=f\circ R_x$ is $C^2$, we see that the conditions to apply Riemmanian Local-Backtracking GD are fully satisfied. 
\label{Example2}\end{example}

{\bf Definition (Riemannian New Q-Newton's method)} Let $X$ be a Riemannian manifold of dimension $m$ with a Strong local retraction $r,R$. Let $f:X\rightarrow \mathbb{R}$ be a $C^2$ function. We choose a real number $1<\alpha $ and randomly $m+1$ real numbers $\delta _0,\ldots , \delta _m$.  We also choose a strictly increasing sequence $\{\gamma _j\}_{j=0,1,2,\ldots } $, for which $\gamma _0=0$, $\gamma _1=1$, $\lim _{j\rightarrow\infty}\gamma _j=\infty$ and $\liminf _{j\rightarrow\infty}\gamma _j/\gamma _{j+1}>0$.  The update rule for Riemannian New Q-Newton's method is as follows. We choose an initial point $x_0$, and construct the sequence $\{x_n\}$ as follows: 

- If $grad(f)(x_n)=0$, then STOP. Otherwise,

- Choose $j$ to be the smallest number among $\{0,1,\ldots ,m\}$ so that $A_n=Hess(f)(x_n)+\delta _j||grad (f)(x_n)||^{\alpha}Id$ is invertible. 

- Let $V_{A_n}^+\subset T_{x_n}X$ be the vector space generated by eigenvectors with {\bf positive} eigenvalues of $A_n$, and $V_{A_n}^-\subset T_{x_n}X$ be the vector space generated by eigenvectors with {\bf negative} eigenvalues of $A_n$. Let $pr_{\pm,A_n}:T_{x_n}X\rightarrow V_{A_n}^{\pm}$ be the corresponding orthogonal projections.  

- Define $w_n$ by the formula $w_n=A_n^{-1}.grad(f)(x_n)$. 

- Let $v_n=pr_{+,A_n}.w_n-pr_{-,A_n}.w_n$. 

- Choose $j$ to be the smallest number so that $\gamma _j r(x_n)/2\leq v_n<\gamma _{j+1}r(x_n)/2$, then define  $\lambda _n=1/\gamma _{j+1}$, and $x_{n+1}=R_{x_n}(-\lambda _nv_n)$. (If $r(x)=\infty$, then we simply choose $\lambda _n=1$.)

\begin{example} This definition was given in  \cite{truong-etal} in the Euclidean setting, which makes precise some folklore heuristics in the Optimisation and Deep Learning communities (see for example \cite[Section 6]{dauphin-pascanu-gulcehre-cho-ganguli-bengjo} for a discussion). There, $\lambda _n=1$ because $r(x_n)=\infty$ for all $x_n$.   

In our definition here, if $||v_n||$ is small (relative to $r(x_n)$), then $\lambda _n=1$. 

The assumption that $\delta _0,\delta _1,\ldots, \delta _m$ are randomly chosen is only needed to make sure that the preimages of sets with {\bf zero} Lebesgue measure, by the dynamical systems associated to Riemannian New Q-Newton's method, are again of {\bf zero} Lebesgue measure (in the literature, this property is sometimes called Lusin ($N^{-1}$) property). This assumption was overlooked in \cite{truong-etal}. On the other hand, experiments in that paper show that the algorithm works well even if we do not choose  $\delta _0,\delta _1,\ldots, \delta _m$ randomly. (See Theorem \ref{TheoremRandomnessLambda} for more detail on the level of randomness required.)  It is possible that indeed Lusin ($N^{-1}$) property can hold for the dynamical systems in Riemannian New Q-Newton's method under much more general assumptions on $\delta _0,\ldots ,\delta _m$.         
\label{Example3}\end{example} 

Besides the discussion in Example \ref{Example3}, we need the following stronger assumption on the Strong local retraction $r,R$, in order to prove some theoretical results for Riemannian New Q-Newton's method. 

{\bf Definition (Real analytic-like Strong local retraction)} We assume that the Strong local retraction $(r,R)$ on $X$ has the following property. For every point $x_0\in X$, there is a small open neighbourhood $U$ of $x_0$ and a small open set $W\subset \mathbb{R}$ , so that if $u(x,s):~U\times W\rightarrow \bigcup _{x\in U}B(x,r(x))$ is continuous, as well as a $C^1$ map in the variable $x$ and a real rational function in the variable $s$, then for all $y\in U$, there is a real analytic function $h$  and a real rational function $\kappa$ such that $h\circ \kappa (s)$ restricts to $\det (grad _y(R_y(u(y,s))))$ on $W$. Here $\det (.)$ is the determinant of a matrix.

\begin{example} The use of this condition lies in that if $s\mapsto \det (grad _y(R_y(u(y,s))))$ ($s\in W$)  is not the zero function identically, then its zero set has Lebesgue measure $0$. 

One prototype example for this is when $X$ is an {\bf open subset} of a Euclidean space and $R$ is the exponential map (which is just the identity map in this case). In this case, we just need to take $h\circ \kappa (s)$ to be the rational function which defines the function $s\mapsto  \det (grad _y(R_y(u(y,s))))$ for $s\in W$. 

More generally, this is the case if the Strong local retraction $R(x,v)$ ($x\in X,v\in B(x,r(x))$) can be extended real analytically (need not be globally a diffeomorphism onto its image) to the whole $T_xX$. For example, this is the case when $X$ is an open subset of a  Lie group, and $R_x$ is the exponential map.   
\label{Example4}\end{example}

\subsection{Main results and proofs}\label{SubsectionMainResults} In this subsection we state and prove main results on convergence and/or avoidance of saddle points results for the algorithms defined in the previous subsection. In addition, we will also prove similar results for a continuous version of Backtracking GD, generalising Theorem 1.1 in \cite{truong}. We will start with some preparation results.  

The following is used to prove properties of the Riemannian Backtracking GD, Riemannian Continuous Backtracking GD as well as Riemannian Local Backtracking GD. We say that a sequence $\{x_n\}$, in a metric space $X$, diverges to infinity if it eventually leaves every compact subsets of $X$. We say that a point $x\in X$ is a cluster point of $\{x_n\}$ if there exists a subsequence $\{x_{n_j}\}$ converging to $x$.  

\begin{theorem} Let $X$ be a Riemannian manifold of finite dimension. Let $d_X$ be the induced metric on $X$. Let $\{x_n\}$ be a sequence in $X$ such that $\lim _{n\rightarrow\infty}d_X(x_{n+1},x_n)=0$. Let $C$ be the set of cluster points of $\{x_n\}$. Let $A$ be a closed subset of $X$. Assume that $C\subset A$. 
Let $B$ be a connected component of $A$, and assume further that $B$ is compact. 

1) Assume that $C\cap B\not=\emptyset$. Then $C\subset B$ and $C$ is connected. 

2) (Capture theorem) Assume that $B$ is a point and $C\cap B\not= \emptyset$. Then $C=B$, i.e. the sequence $\{x_n\}$ converges to the point $B$. 

3) Assume that $A$ is at most countable. Then either $\{x_n\}$ converges to a point, or $\{x_n\}$ diverges to infinity.  

\label{TheoremConvergenceSequenceRiemannianManifolds}\end{theorem}
\begin{proof}
The idea is to apply Nash's embedding theorem to the arguments in \cite{truong-nguyen}.

By Theorem \ref{TheoremNashEmbedding}, for the purpose of comparing metrics, we can assume that $X$ is a Riemannian submanifold of some Euclidean space $\mathbb{R}^N$. We let $||.||$ denote the usual norm on $\mathbb{R}^N$. We let $\mathbb{P}^N$ be the real projective space of dimension $N$, and $d(.,.)$ the standard spherical metric.

Then for $x,x'\in X$, we have the following inequalities: $d_X(x,x')\geq ||x-x'||\geq d(x,x')$. The first inequality follows since $X$ is a Riemannian submanifold of $\mathbb{R}^N$.  The second inequality is probably well known, and a detailed proof is given in \cite{truong-nguyen}.  Hence, we also have $\lim _{n\rightarrow\infty}d(x_{n+1},x_n)=0$. Note that while the metric $d_X$ may be very different from that of the restriction of $d(.,.)$ to $X$ (in particular, since $d(.,.)$ is bounded, while $d_X$ may not), $X$ is a {\bf topological} subspace of $\mathbb{P}^N$. In particular,  convergence behaviour of a sequence $\{x_n\}\subset X$, considered in either the original topology on $X$ or the induced one from $\mathbb{P}^N$, is the same. 

We let $C'\subset \mathbb{P}^N$ be the set of cluster points of $\{x_n\}$, considered as a sequence in $\mathbb{P}^N$. Then $C'$ is the closure of $C$ in $\mathbb{P}^N$, and $C'\cap \mathbb{R}^N=C\subset A$. Since $\mathbb{P}^N$ is a compact metric space, it follows from \cite{asic-adamovic} that $C'$ is a connected set.  

 1) If $C'\cap B=C\cap B$ is not empty, then since $C'\cap \mathbb{R}^N\subset A$, $C'$ is connected and $B$ is a compact connected component of $A$, it follows that $C'\subset B$. Hence, $C'=C'\cap B=C\cap B=C$ is connected. 
 
 2) From 1) we have that $C\subset B$. If $B$ is  a point, then we must have $C=B$, which means that $\{x_n\}$ converges to the point $B$. 
 
 3) Since $A$ is at most countable, any connected component of $A$ is 1 point, and hence must be compact. If $C\not= \emptyset$, then $C$ must intersect at least one of the points in $A$, and hence by part 2) the whole sequence $\{x_n\}$ must converge to that point. Otherwise, $C=\emptyset$, and hence in this case $\{x_n\}$ diverges to infinity by definition.

\end{proof}

The following is used to prove properties of the Riemannian New Q-Newton's method algorithm.  

\begin{theorem} Let $X$ be a Riemannian manifold. Let $f:X\rightarrow \mathbb{R}$ be a $C^2$ function. Define $U=X\backslash \{x\in X:~grad(f)(x)=0\}$. For $x\in U$ and $\delta \in \mathbb{R}$ we define $A(x,\delta )=Hess(f)(x)+\delta ||grad (f)(x)||^{\alpha} Id$.

1) There exists a set $\Delta \subset \mathbb{R}$ of Lebesgue measure $0$ such that:  For all $\delta \in \mathbb{R}\backslash \Delta$, the set $\{x\in U:~A(x,\delta )$ is not invertible $\}$ has Lebesgue measure $0$. 

2) Fix $\delta \in \mathbb{R}\backslash \Delta$. We define $w(x,\delta )=A(x,\delta )^{-1}.grad (f)(x)$ and $v(x,\delta )=pr_{+,A(x,\delta )}.w(x,\delta )-pr_{-,A(x,\delta )}w(x,\delta )$, and $\lambda (x,\delta )=r(x)/||v(x,\delta )||$. Let $U_{\delta}=U\backslash \{x\in U:~p(x,\delta )=0\}$.  Then $v(x,\delta ),\lambda (x,\delta )$ are continuous on $U_{\delta}$. If we assume moreover that $f$ is $C^3$, then $v(x,\delta )$ is $C^1$ on $U_{\delta}$, and $\lambda (x,\delta )$ is  $C^1$ on $U_{\delta}\backslash \{x:~r(x)=||v(x,\delta )||\}$. 

3) Let the setting be as in part 2). Moreover, assume that $R$ is real analytic - like. Define $H(x,\delta )=R_x(-\lambda (x,\delta )v(x,\delta ))$, here $\lambda (x,s)$ is as in the definition for Riemannian New Q-Newton's method. There is a set $\Delta '\subset \mathbb{R}\backslash \Delta $ of Lebesgue measure $0$ so that for all $\delta \in \mathbb{R}\backslash (\Delta \cup \Delta ')$, the set $\{x\in U_{\delta }:~r(x)\not= ||v(x,\delta )||,~grad(H)(x,\delta )$ is not invertible$\}$ has Lebesgue measure $0$. 

4) Let the setting be as in part 3). There is a set $\Delta "\subset \mathbb{R}$ of Lebesgue's measure zero so that the following is satisfied. Let $\mathcal{E}\subset X$ be a set of Lebesgue measure $0$. If $\delta \in \mathbb{R}\backslash \Delta "$, then $H(.,\delta )^{-1}(\mathcal{E})\subset X$ has Lebesgue's measure $0$.

\label{TheoremRandomnessLambda}\end{theorem}

\begin{proof}
Because we are working locally, we can assume that $X$ is a Euclidean space and $R_x(v)=x+v$ (this is the first order approximation of $R$, which is the only thing we need when checking that the gradient of the dynamics is invertible). 

1) Define $p(x,\delta )=\det (A(x,\delta ))$.  For $x\in U$, then $p(x,\delta )$ is a polynomial in $\delta$ of degree exactly $m=\dim (X)$. Note that $A(x,\delta )$ is not invertible if and only if $p(x,\delta )\not= 0$. We consider the set $\Gamma =\{(x,\delta )\in U\times \mathbb{R}:~p(x,\delta )=0\}$. Then $\Gamma $ is a measurable subset of $U\times \mathbb{R}$, and we need to show the existence of a set $\Delta \subset \mathbb{R}$ of Lebesgue measure $0$ so that for all $\delta \in \mathbb{R}\backslash \Delta$, the set $\Gamma _{\delta}=\{x\in U:~p(x,\delta )=0\}$ has Lebesgue measure $0$. 

First, we will show that $\Gamma$ has Lebesgue measure $0$. Let $1_{\Gamma}:~U\times \mathbb{R}\rightarrow \{0,1\}$ be the characteristic function of $\Gamma$.  For a set $A$, we denote by $|A|$ its Lebesgue measure. Then, by Fubini-Tonneli's theorem in integration theory for non-negative functions, one obtains: 
\begin{eqnarray*}
|\Gamma |=\int _{U\times \mathbb{R}}1_{\Gamma}(x,\delta ) d(x,\delta ) =\int _{U}(\int _{\mathbb{R} }1_{\Gamma}(x,\delta )d \delta )dx=0.
\end{eqnarray*}
This is because for each $x\in U$, the set $\{\delta \in \mathbb{R}:~p(x,\delta )=0\}\subset \mathbb{R}$ has cardinality at most $m$, and hence has Lebesgue measure $0$. 

Applying Fubini-Tonneli's theorem again, but now doing the integration on $U$ first, one obtains that for a.e. $\delta \in \mathbb{R}$, the set $\Gamma _{\delta}=\{x\in U:~p(x,\delta )=0\}\subset U$ has Lebesgue measure $0$.  

2) The claims follow from perturbation theory for linear operators \cite{kato}, where $pr_{\pm ,A(x,\delta )}$ can be represented via Cauchy's integrations on the complex plane containing the variable $\delta$. The readers can see for example \cite{truong-etal} on details how the arguments go. 

3)  For each $x\in U$, define $\Delta _x=\{\delta \in \mathbb{R}:~p(x,\delta )=0\}$. Then, $\Delta _x$ is a finite set (zeros to the polynomial $p(x,\delta )$).

By calculating, we see that when it is legit (c.f. part 2 above), then $grad(H(.,\delta ))$ is invertible at $x$ iff $q(x,\delta )=det (grad (H(.,\delta )))\not= 0$. 

Now, we fix $x_0\in U$ and $\lambda _0\in \Delta _{x_0}$. We note that there is an open interval $(a,b)\subset \mathbb{R}$ containing $\delta _0$, on which $grad (H(x_0,\delta ))$  is legit and $-\lambda (x,\delta )v(x,\delta )$ is a rational function of $\delta$. Here is how to see this. We can write $\delta =\delta _0+\epsilon$, where $\epsilon$ is a small real number. Then, for $x$ close to $x_0$ and $\delta$ close to $\delta _0$, the eigenvalues of $A(x,\delta )$ are close to those of $A(x_0,\delta _0)$. Therefore, we can choose two close curves $\gamma ^+$ and $\gamma ^-$ in the complex plane containing the variable $\epsilon$, so that for $\delta =\delta _0+\epsilon$ close to $\delta _0$
\begin{eqnarray*}
v(x,\delta _0+\epsilon )&=& \frac{1}{2\pi i}\int _{\gamma ^+} ((A(x,\delta _0)+\epsilon ||grad(f)(x) ||^{\alpha})^{-1}-\zeta )^{-1}d\zeta \\
&&-\frac{1}{2\pi i}\int _{\gamma ^-} ((A(x,\delta _0)+\epsilon ||grad(f)(x) ||^{\alpha})^{-1}-\zeta )^{-1}d\zeta .
\end{eqnarray*}
 
Since we assumed that $||v(x_0,\delta _0)||\not= r(x_0)$, it follows that $||v(x,\delta _0+\epsilon )||\not= r(x)$ under the current assumptions. So we have that $\lambda (x,\delta _0+\epsilon )$ has the same form in the considered domain. Thus we can assume that $\lambda (x,\delta _0+\epsilon )=\gamma _{j+1}$ for a constant $j$. Moreover, the above formula for $v(x,\delta _0+\epsilon )$ clearly shows that it is a rational function on $\epsilon$, of degree exactly $m$. 

Hence, by the assumption that $R$ is real analytic-like, it follows that (see the discussion in Example \ref{Example4}), either $q(x_0,\delta _0+\epsilon )$ is zero identically, or its set of zero has Lebesgue's measure zero. We will show that the former case cannot happen. Indeed, while we defined the function 
\begin{eqnarray*}
\epsilon \mapsto \det (grad (R(-\lambda _jv(x,\delta _0+\epsilon ))))
\end{eqnarray*}
only for $\epsilon $ small numbers, it is no problem to extend its domain of definition to the maximum domain  $W\subset \mathbb{R}$ for which $-\lambda _jv(x,\delta _0+\epsilon )\in B(x,r(x)) $. The same function $h\circ \kappa (s)$ will work for the whole domain $W$. 

Now, we note that for $\epsilon $ large enough (uniformly in $x$), then $-\lambda _{j+1}v(x,\delta _0+\epsilon \in B(x,r(x))$. It can be seen from the above integral representation for $v(x,s)$. Here is another way, easier to see, way. We check this claim for example at the point $x_0$. Let  $\zeta _1,\ldots ,\zeta _l >0$ and $\zeta _{l+1},\ldots ,\zeta _{m}<0$ be all eigenvalues of $A(x_0,\delta _0)$. Let also $e_1,\ldots ,e_l$ and $e_{l+1},\ldots ,e_m$ be the corresponding orthonormal basis. If $grad(f)(x_0)=a_1e_1+\ldots +a_me_m$, then 
\begin{eqnarray*}
v(x_0,\delta _0+\epsilon )=\sum _{\beta <l+1}\frac{a_{\beta}}{\zeta _{\beta } +\epsilon ||grad(f)(x_0)||^{\alpha }}e_{\beta}- \sum _{\beta >l}\frac{a_{\beta}}{\zeta _{\beta } +\epsilon ||grad(f)(x_0)||^{\alpha }}e_{\beta}. 
\end{eqnarray*}
Hence $\lim _{\epsilon\rightarrow\infty}-\lambda _{j+1}v(x_0,\delta _0+\epsilon )=0$, which confirms the claim.

Now, we are ready to show that $q(x_0,\delta _0+\epsilon )$ is not zero identically, for $\epsilon $ small enough. Indeed, since $h\circ \kappa (\delta _0+\epsilon )$ restricts to $q(x_0,\delta _0+\epsilon )$, if $q(x_0,\delta _0+\epsilon )$ is zero identically, then the same would be true for $h\circ \kappa (\delta _0+\epsilon )$. However,  we will show that this is not the case. For $\epsilon $ large enough, from what written, we get that: $h\circ \kappa (\delta _0+\epsilon )=grad _x(R_x(-\lambda _{j+1}v(x,\delta _0+\epsilon )))$. On the other hand, by Chain's rule: 
\begin{eqnarray*}
grad _x(R_x(-\lambda _{j+1}v(x,\delta _0+\epsilon )))=(grad _xR_x|_{-\lambda _{j+1}v(x,\delta _0+\epsilon ))}).grad _x(-\lambda _{j+1}v(x,\delta _0+\epsilon ))).
\end{eqnarray*}
By explicit calculation, one then sees that $\lim _{\epsilon\rightarrow\infty}grad _x(R_x(-\lambda _{j+1}v(x,\delta _0+\epsilon )))=Id$. (One can check this readily for the Euclidean case, in which case $H(x,\lambda )=x-\lambda _{j+1}v(x,\delta  )$, using either representation of $v(x,\delta )$ above.) Hence, in particular $\lim _{s\rightarrow \infty}h\circ \kappa (\delta _0+\epsilon )=1$, which shows that $h\circ \kappa (\delta _0+\epsilon )\not= 0$, as desired. 

From the representation of $v(x_0,\delta _0+\epsilon )$, we see that $$||v(x,s)||^2=\sum _{\beta }\frac{a_{\beta }^2}{(\zeta _{\beta }+\epsilon ||grad (f)(x_0)||^{\alpha})^2},$$ and hence clearly is not a constant. (For example, one can see that it is a rational function with at least one pole, but is not identically $\infty$, because the limit when $\epsilon$ goes to $\infty$ is $0$.) Therefore, the set of $\delta \in \mathbb{R}$ for which $||v(x,s)||=r(x)$ is at most a finite set. 

Combining all the above, here is a summary of what we proved so far: Except for a countable set of values $\delta \in \mathbb{R}$, the map $x\mapsto R_x(-\delta (x,\delta )v(x,\delta ))$ is $C^1$ and is a local diffeomorphism.  With this, now we finish the proof of this part 3). As in the proof of part 1), we define the set $\Gamma '$ to consist of all pairs $(x,\delta )$ for which $A(x,\delta )$ is invertible but the gradient of $H(x,\delta )$ is not invertible. The slices $\Gamma '_{x}=\{\delta :~(x,\delta )\in \Gamma '\}$ have been shown above to have Lebesgue's measure zero. Hence, by the use of Fubini-Tonelli's theorem, we get that the slices $\Gamma '_{\delta}=\{x:~(x,\delta )\in \Gamma '\}$ also have Lebesgue's measure zero for $\delta \in \mathbb{R}$ outside of a set of Lebesgue's measure $0$. 

 4) This is a consequence of part 3 and the following well-known fact: 
 
 {\bf Fact.} Let $H:U\rightarrow V$ be a map, where $U,V$ are open subsets in $\mathbb{R}^m$. Assume that there is a set $\mathcal{E}'\subset U$ of Lebesgue's measure $0$ so that for all $x\in U\backslash \mathcal{E}'$ then $H$ is $C^1$ near $x$ and $grad (H)$ is invertible there. Then, if $\mathcal{E}\subset V$ has Lebesgue's measure $0$, it follows that $H^{-1}(\mathcal{E})\subset U$ has Lebesgue's measure $0$.  

\end{proof}

Now we are ready to state and prove convergence and/or avoidance of saddle points for the mentioned algorithms. The first result is for Riemannian Backtracking GD. 

\begin{theorem} Let $X$ be a Riemannian manifold and $r,R$ a Strong local retraction on $X$. Let $f:X\rightarrow \mathbb{R}$ be a $C^1$ function. Let $x_0\in X$, and $\{x_n\}$ the sequence constructed by the Riemannian Backtracking GD algorithm. Then,

1) If $x_{\infty}\in X$ is a cluster point of $\{x_n\}$, then $grad(f)(x_{\infty})=0$. 

2) Either $\lim _{n\rightarrow\infty}f(x_n)=-\infty$, or $\lim _{n\rightarrow\infty}d_X(x_n,x_{n+1})=0$. 

3) Let $A$ be the set of critical points of $f$, and let $B\subset A$ be a compact connected component. Let $C$ be the set of cluster points of $\{x_n\}$. If $C\cap B\not= \emptyset$, then $C\subset B$ and $C$ is connected. 

4) (Capture theorem.) Let the setting be as in part 3. If $B$ is a point, then $C=B$.

5) (Morse's function.) Assume that $f$ is a Morse function. Then the sequence $\{x_n\}$ either converges to a critical point of $f$, or diverges to infinity. Moreover, if $f$ has compact sublevels, then the sequence $\{x_n\}$ converges.

\label{Theorem1}\end{theorem} 
\begin{proof}
1) As in the Euclidean case (see e.g. \cite{bertsekas, truong-nguyen}), by using Armijo's condition and Lemma \ref{LemmaLowerBound}. Here, one subtlety to note is that, by the properties of a Strong local retraction, if the sequence $\{x_{n_j}\}$ converges to $x_{\infty}$, then $||-\delta _{n_j}grad(f)(x_{n_j})||<3r(x_{\infty})/4$ when $j$ is large enough. Therefore, we can apply Armijo's condition for $f(x_{n_j+1})-f(x_{n_j})$ to obtain a contradiction if $\limsup _{j\rightarrow\infty}||grad(f)(x_{n_j})||$ were not zero.  

2) We use here the fact $d_X(x_n,x_{n+1})$ is bounded by $||\delta _n grad(f)(x_n)||$, and proceed as in the Euclidean case. Note that here $\delta _n$ is also allowed to not uniformly bounded,  see \cite{truong4}. 

3) If $B$ is non-empty, then $\lim _{n\rightarrow\infty}f(x_n)>-\infty$. Hence, by 2) we have $\lim _{n\rightarrow\infty}d_X(x_n,x_{n+1})=0$. Thus, Theorem \ref{TheoremConvergenceSequenceRiemannianManifolds} applies to give the desired conclusion. 

4) and 5) Follows from 3), as in the proof of Theorem \ref{TheoremConvergenceSequenceRiemannianManifolds}.

\end{proof}

In the literature, usually Capture theorem is stated under the assumption that $B$ is an isolated local minimum. Here we see that the isolated assumption alone is sufficient. The next result is for Riemannian Local Backtracking GD. 

 \begin{theorem}
Let $f:X\rightarrow \mathbb{R}$ be a $C^{1}$ function which satisfies the condition in the definition for Riemannian Local Backtracking GD. Assume moreover that $\nabla f$ is $C^2$ near its generalised saddle points. For any $x_0\in X$, we construct the sequence $\{x_n\}$ as in the update rule for Riemannian Local Backtracking GD. Then: 

(i) For all $n$ we have $f(x_n-\widehat{\delta} (x_n)\nabla f(x_n))-f(x_n)\leq -(1-\alpha ) \widehat{\delta} (x_n)||\nabla f(x_n)||^2$. 

(ii) For every $x_0\in X$, the sequence $\{x_{n}\}$ either satisfies $\lim _{n\rightarrow\infty} d_X(x_n,x_{n+1}=0$ or diverges to infinity. Each cluster point of $\{x_n\}$ is a critical point of $f$. a) If moreover, $f$ has at most countably many critical points, then $\{x_n\}$ either converges to a critical point of $f$ or diverges to infinity. b) More generally, if $A$ is a compact component of the set of critical points of $f$ and $B$ is the set of cluster points of $\{x_n\}$ and $A\cap B\not= \empty$, then $B\subset A$ and $B$ is a connected set. 
 
 (iii) For {\bf random choices} of $\delta _0, \alpha $ and $\beta$, there is a set $\mathcal{E}_1\subset X$ of Lebesgue measure $0$ so that for all $x_0\in X\backslash \mathcal{E}_1$, if the sequence $\{x_n\}$ converges, then the limit point cannot be a generalised saddle point. 
 
(iv) For {\bf random choices} of $\delta _0, \alpha $ and $\beta$, there is a set $\mathcal{E}_2\subset X$ of Lebesgue measure $0$ so that for all $x_0\in X \mathcal{E}_2$, any cluster point of the sequence $\{x_{n}\}$ cannot be a saddle point, and more generally cannot be an isolated generalised saddle point. 
\label{Theorem2}\end{theorem}
\begin{proof}
The proof is similar to that of its Euclidean counterpart in \cite{truong}. We recall here the main points. 

First, the conditions in the Riemannian Local Backtracking GD imply that the corresponding dynamical systems $H(x)$ is locally one of a finite number (this number can vary depending on the open set in consideration) of maps which have bounded torsion $d_X(H(x),H(y))\geq Cd_X(x,y)$, for some constant $C$ depending on the considered open set. This implies that the inverse of a set of Lebesgue measure $0$ by $H$ will also have Lebesgue measure $0$.  

Second, by the trick of using Lindel\"off lemma as in \cite{panageas-piliouras}, one reduces to consider small neighbourhoods of at most countably many generalised saddle points $\{z_i\}$. Then the randomness of the hyperparmeters is to ensure that for all $i$ the number $\alpha /L(z_i)$ does not belong to the set $\{\beta ^n\delta _0:~n=0,1,2,\ldots \}$ This then implies that near a point $z_i$, since $||grad(f)||$ will be small, the learning rate will be a constant. Thus, the dynamical system $x\mapsto H(x)$ is $C^1$ near $z_i$, which allows us to use Stable - Center manifold theorem in classical Dynamical systems theory \cite{shub}.   
\end{proof}

Similarly, we have a continuous version for Riemannian Backtracking GD.  We remark that our construction is different from that of using ODE (such as in gradient flow) in the literature. 

\begin{theorem}
Let $f:X\rightarrow \mathbb{R}$ be a $C^{1}$ function, so that $\nabla f$ is locally Lipschitz continuous. Assume moreover that $f$ is $C^2$ near its generalised saddle points. Then there is a smooth function $h:X\rightarrow (0,\delta _0]$ so that the map $H:X\rightarrow \mathbb{R}^k$ defined by $H(x)=R_x(-h(x)\nabla f(x))$ has the following property: 

(i) For all $x\in X$, we have $f(H(x)))-f(x)\leq -\alpha h(x)||\nabla f(x)||^2$. 

(ii) For every $x_0\in X$, the sequence $x_{n+1}=H(x_n)$ either satisfies $\lim _{n\rightarrow\infty}d_X(x_n,x_{n+1})=0$ or diverges to infinity. Each cluster point of $\{x_n\}$ is a critical point of $f$. a) If moreover, $f$ has at most countably many critical points, then $\{x_n\}$ either converges to a critical point of $f$ or diverges to infinity. b) More generally, if $A$ is a compact component of the set of critical points of $f$ and $B$ is the set of cluster points of $\{x_n\}$ and $A\cap B\not= \empty$, then $B\subset A$ and $B$ is a connected set. 

(iii) There is a set $\mathcal{E}_1\subset X$ of Lebesgue measure $0$ so that for all $x_0\in X\backslash \mathcal{E}_1$,  the sequence $x_{n+1}=H(x_n)$, {\bf if converges}, cannot converge to a {\bf generalised} saddle point.  

(iv) There is a set $\mathcal{E}_2\subset X$ of Lebesgue measure $0$ so that for all $x_0\in X\backslash \mathcal{E}_2$, any cluster point of the sequence $x_{n+1}=H(x_n)$ is not a saddle point, and more generally cannot be an isolated generalised saddle point. 
\label{Theorem3}\end{theorem}
\begin{proof}
Again, the proof is similar to its Euclidean counterpart in \cite{truong}. The main point is to use the stronger form of Lindel\"off lemma that, since $X$ is a Riemannian manifold of finite dimension $m$, each open covering of $X$ has a subcovering which is locally finite. The latter means that for every $x$, there is a small open subset $U$ around $x$, so that every point $y\in U$ belongs to at most $m+1$ sets in the subcovering.  We then use this fact and partition of unity to cook up a smooth function for learning rates $h:X\rightarrow (0,\delta _0]$ which has the same properties as that of the learning rates in the proof of Theorem \ref{Theorem3}. When having this, we can proceed as before. 
\end{proof} 

Finally, we state and prove the result for Riemannian New Q-Newton's method. 

\begin{theorem} Let $f:X\rightarrow \mathbb{R}$ be a $C^3$ function. Let $\{x_n\}$ be a sequence constructed by the Riemannian New Q-Newton's method. Assume that $\{x_n\}$ converges to $x_{\infty}$. Then

1) $grad (f)(x_{\infty})=0$, that is $x_{\infty}$ is a critical point of $f$.

2) Assume that $r,R$ satisfies the Real analytic-like condition.  There is a set $\mathcal{A}\subset X$ of Lebesgue measure $0$, so that if $x_0\notin \mathcal{A}$, then $x_{\infty}$ cannot be  a saddle point of $f$. 

3) Assume that $r,R$ satisfies the Real analytic-like condition. If $x_0\notin \mathcal{A}$ (as defined in part 2) and $Hess(f)(x_{\infty})$ is invertible, then $x_{\infty}$ is a local minimum and the rate of convergence is quadratic. 

4) More generally, if $Hess(f)(x_{\infty})$ is invertible (but no assumption on the randomness of $x_0$), then the rate of convergence is at least linear. 

5) If $x_{\infty}'$ is a non-degenerate local minimum of $f$, then for initial points $x_0'$ close enough to $x_{\infty}'$, the sequence $\{x_n'\}$  constructed by Riemannian New Q-Newton's method will converge to $x_{\infty}'$. 
\label{Theorem4}\end{theorem} 
\begin{proof}
The proof is similar to that of the Euclidean counterpart in \cite{truong-etal}. We note that in parts 2) and 3), if the Strong local retraction $r,R$ is not required to satisfy the Real analytic-like condition, then one still can prove that there are local Stable-Center manifolds for the associated dynamical systems near saddle points. Hence, on all Riemannian manifolds one has local guarantee for avoidance of saddle points near saddle points. The Real analytic-like condition is needed to assure global avoidance of saddle points, via the use of Theorem \ref{TheoremRandomnessLambda} (as mentioned, the fact that the numbers $\delta _0,\ldots ,\delta _m$ should be random was overlooked in \cite{truong-etal}). Parts 4) and 5), which are local in nature, can be proven exactly as in \cite{truong-etal}. 

There is a subtle point in the proof of part 1), compared to its Euclidean counterpart, which lies in the fact that $\lambda _n $ in the update rule is in general not the constant $1$. This is because in general we do not have global retractions, for example if we work with open subsets of complete Riemannian manifolds. The choice of $\lambda _n$ in our update rule, which is about $1/||v_n||$ when $||v_n||$ is large, is important to assure that $grad (f)(x_{\infty})=0$. If $\delta _n$ has another asymptotic growth, such as $1/||v_n||^2$, then there is no such guarantee.  Here we give a detailed proof of 1) to illustrate the point. 

Proof of part 1):  Since $x_{n+1}=R_{x_n}(-\lambda _nv_n)$ converges to $x_{\infty}$, and $||\lambda _nv_n||\leq r(x_n)/2$ for all $n$, together with the fact that $R_{x_{\infty}}(v)$ is a diffeomorphism for $||v||<r(x_{\infty})$, it follows that $\lim _{n\rightarrow\infty}||-\lambda _nv_n||=0$. This implies that first of all, $||v_n|| $ is bounded, since if it were true that $\lim _{n\rightarrow\infty}||v_n||=\infty$, then for large $n$ we would have $\lambda _n\sim r(x_n)/2||v_n||$, and hence we would have a contradiction that $\lim _{n\rightarrow\infty}||\lambda _nv_n||>0$. Therefore, $\lambda _n$ is uniformly bounded from below by a positive number, thus from $\lim _{n\rightarrow\infty}||-\lambda _nv_n||=0$ we obtains also that $\lim _{n\rightarrow\infty}||v_n||=0$. Since $||v_n||=||w_n||$ for all $n$, we have that $\lim _{n\rightarrow\infty}||w_n||=0$. Then (note that $||A_n||$ is uniformly bounded) $$||grad(f)(x_{\infty})||=\lim _{n\rightarrow\infty}||grad(f)(x_n)||=\lim _{n\rightarrow\infty}||A_n.w_n||=0.$$ 
\end{proof}

\subsection{Some experiments} In this subsection, we present some experiments with singular cost functions or constrained optimisation on Euclidean spaces. In all examples, we choose $R_x(v)=x+v$, while the Riemannian manifold $X$ and the function  $r(x)$ will be changed appropriately according to each example. All the cost functions below here were considered in \cite{truong-etal}, where only the Euclidean algorithms were used. 

In our experiments, we will also use the observation in \cite{truong-etal}, that for proofs of the theoretical results, we only need the matrix $A_n$ to have the form $\nabla ^2f(x_n)+\delta _j||\nabla f(x_n)||^{\alpha}Id$ when $||\nabla f(x_n)||$ is small. When $||\nabla f(x_n)||$ is big, instead of the above formula, we will choose $A_n=\nabla ^2f(x_n)+\delta _jId$. That is, we choose $A_n=\nabla ^2f(x_n)+\delta _j \min \{||\nabla f(x_n)||^{\alpha},1\}Id$ for all $n$. 

Yet another simplification. For the functions considered below, we see that if $\nabla ^2f(x_n)$ is not invertible, then $\nabla ^2f(x_n)+\min \{||\nabla f(x_n)||^{\alpha},1\}Id$ is invertible. Therefore, even in the theorems we need to choose $m+1$ random numbers, we choose only $2$ numbers here: $\delta _0=0$ and $\delta _1=1$. Also, for all the experiments below we choose $\alpha =2$.

Besides comparing Riemannian New Q-Newton's method with New Q-Newton's method, we will also compare them with some other versions of Newton's method: the original version of Newton's method, BFGS and Random Newton's method (whose update rule is $x_{n+1}=x_n-\kappa _n(\nabla ^2f(x_n))^{-1}.\nabla f(x_n)$ where $\kappa _n$ is randomly chosen in $(0,2)$), as well as the obvious Riemannian version of Newton's method. Since we use the code for BFGS from the python's library, we do not consider its Riemannian version. Also, we choose the sequence $\gamma _j$ (in the definition of Riemannian New Q-Newton's method) to be $\gamma _j=j$. 

{\bf Remark.} In some of Examples 7-9, which are constrained optimisation, while all the above mentioned modifications of Newton's method don't work well (again, Riemannian New Q-Newton's method is the best performance among these), we will see that Riemannian Backtracking GD works pretty well. 

\underline{\bf Case 1: Singular cost functions.} The cost function is continuous but may not be differentiable on a small exceptional closed set $E$ of Lebesgue measure $0$. 

In this case, since we don't expect a point in the usual New Q-Newton's method to fall into the exception set $E$ (since $E$ is small), we can apply New Q-Newton's method with a random initial point $x_0$. Alternatively, we can think as the cost function is defined on the Riemannian manifold $X$ which is the complement of $E$. Then for $r(x)$ we will choose the distance from $x$ to $E$, that is $r(x)=\inf _{y\in E}||x-y||$. 

{\bf Example 1:} The function $f(t)=|t|^{1+0.3 }$. Here $f$ is $C^1$, but is not $C^2$ at $0$. It has a global minimum at $0$, and no critical point except $0$. We take $X=\mathbb{R}\backslash \{0\}$, and $r(x)=|x|=$ the distance from $x$ to the singular point $0$. The initial point is $1.00001188$. 

Newton's method: after $50$ steps, arrives at $2.5e+18$. 

New Q-Newton's method: after $50$ steps, arrives at $2.5e+18$. 

Random Newton's method: after $50$ steps, arrives at $-89897510364.32666$. 

BFGS: after 37 steps, arrives at $-1.81643778e-34$. (Good)

Riemannian Newton's method: after $38$ steps, arrives at $1.11e-11$. (Good)

Riemannian New Q-Newton's method: after $38$ steps, arrives at $1.11e-11$. (Good)

{\bf Example 2:} The function $f(t)=|t|^{0.3}$. This function is continuous, but is not $C^1$ at $0$. It has a global minimum at $0$, and no critical point except $0$. We take $X=\mathbb{R}\backslash \{0\}$, and $r(x)=|x|=$ the distance from $x$ to the singular point $0$. The initial point is $1.00001188$.

Newton's method: after $50$ steps, arrives at $1.85e+19$. 

New Q-Newton's method: after $50$ steps, arrives at $4e-19$. (Good)

Random Newton's method: after $50$ steps, arrives at $5.01e+16$. 

BFGS: get the error message: "Desired error not necessarily achieved due to precision loss."  

Riemannian Newton's method: after $50$ steps, arrives at $286500321.9965359$. 

Riemannian New Q-Newton's method: after $50$ steps, arrives at $9.09e-15$. (Good)

{\bf Example 3:} The function $f(t)=e^{-1/t^2}$. This function is smooth everywhere. However, formally we can say that it is "undefined" at $0$. It has a global minimum at $0$, has no other critical point, but the gradient converges to $0$ when $|t|\rightarrow \infty$. We take $X=\mathbb{R}\backslash \{0\}$, and $r(x)=|x|=$ the distance from $x$ to the singular point $0$. The initial point is $3$. 

Newton's method: after $50$ steps, arrives at $5529152.050364867$. 

New Q-Newton's method: after $50$ steps, arrives at $0.12959775062453516$. (Good)

Random Newton's method: after $50$ steps, arrives at $10384545.533759983$. 

BFGS: after $2$ iterations, arrives at $-0.11289137$. (Good)

Riemannian Newton's method: after $50$ steps, arrives at $5529152.050364867$. 

Riemannian New Q-Newton's method: after $50$ steps, arrives at $0.13236967077626907$. (Good)

{\bf Example 4:} The function $f(x,y)=x^3 sin(1/x)+y^3sin(1/y)$. This function is $C^1$, but its Hessian is singular at $x=0$ or $y=0$. It has infinitely many local minima and local maxima, together with non-isolated critical points having $x=0$ or $y=0$. We choose $X=\mathbb{R}^2\backslash (\{x=0\}\cup \{y=0\})$, and choose $r(x,y)=\min \{|x|,|y|\}$, which is the distance from a point $(x,y)$ to the boundary of $X$. The initial point is $(-0.99998925, 2.00001188)$. 

Newton's method: after $7$ steps, arrives at $(-0.03039904,  0.0042162)$. The Hessian of $f$ at this point has both positive and negative eigenvalues. Seem to close to a saddle point.  

New Q-Newton's method: after $13$ steps, arrives at $(-0.0236179,   0.00400409)$. The Hessian near this point is positive definite. 
 
Random Newton's method: after $22$ steps, arrives at $(0.13382953, -0.00326482)$. The Hessian near this point is non-definite. 

BFGS: after $9$ iterations, arrives at $(-0.24520924,  0.01820721)$. 

Riemannian Newton's method: after $12$ steps, arrives at $(-0.24520924,  0.24520924)$. The Hessian near this point is positive definite. 

Riemannian New Q-Newton's method: after $12$ steps, arrives at $(-0.24520924,  0.24520924)$. The Hessian near this point is positive definite. 

{\bf Example 5: } The function $f(x,y)=100(y-|x|)^2+|1-x|$. The function is singular at $x=0$ or $x=1$. It has a global minimum at $(1,1)$, and no critical points elsewhere.  We choose $X=\mathbb{R}^2\backslash (\{x=0\}\cup \{x=1\})$, and $r(x,y)=\min \{|x|,|1-x|\}$ which is the distance from a point $(x,y)$ to the boundary of $X$. The initial point is $(0.55134554, -0.75134554)$.

Newton's method:  Error: singular Hessian matrix. 

New Q-Newton's method: after 500 steps, arrives at $(1.81294761, 1.81543514)$. 
 
Random Newton's method: Error: singular Hessian matrix. 

BFGS: after $8$ steps, arrives at $(0.06698203, 0.06382963)$. 

Riemannian Newton's method: Error: singular Hessian matrix. 

Riemannian New Q-Newton's method: after 500 steps, arrives at $(1,0.86409541)$. 

In this case, no method could get close to the global minimum, but we can say that the performance of Riemannian New Q-Newton's method is best. 

{\bf Example 6:}  $f(x,y)=5|x|+y$. In this case, there is no critical point. The function is singular at $x=0$. We take $X=\mathbb{R}^2\backslash \{x=0\}$, and $r(x,y)=|x|=$ the distance from the point $(x,y)$ to the boundary of $X$. The initial point is $(-0.99998925,  2.00001188)$.

Newton's method:   Singular Hessian matrix. 

New Q-Newton's method: after 500 steps, arrives at $( -0.99998925, -497.99998812)$. 
 
Random Newton's method: Singular Hessian matrix. 

BFGS: Error "Desired error not necessarily achieved due to precision loss." 

Riemannian Newton's method: Singular Hessian matrix. 

Riemannian New Q-Newton's method: after $500$ steps, arrives at $(-5.17500998e-147,  1.80001403e+000)$. (Converging to the boundary of $X$.) 

In this example, a good method should diverge to infinity. Hence, in this case New Q-Newton's method is best.

\underline{\bf Case 2: Constrained optimisation.} We consider some constrained optimisation problem of the form $\min _{||x||\leq 1}f(x)$, where $f$ is a quadratic function whose Hessian has at least one negative eigenvalue. We note that without the constraint $||x||\leq 1$, the good behaviour is that the sequence converges to infinity. For some of the examples considered below, none of the above mentioned modifications of Newton's method work well, and we will use in addition Riemannian Backtracking GD. The hyperparameters for Riemannian Backtracking GD will be fixed as follows: $\delta _0=1$, $\alpha =0.5$ and $\beta =0.7$. 

{\bf Example 7:} $\min _{||(x,y)||\leq 1}(f(x,y)=x^2+y^2+4xy)$. In this case, globally the function $f$ has a saddle point at $(0,0)$ and no other critical points. It can be checked easily that the minimum in the domain $||(x,y)||\leq 1$ is obtained at $x=-y=\pm \sqrt{0.5}$ $\sim $ $0.707107$. Indeed, the Hessian of this function has 2 eigenvalues $1$ and $-1$. The vector $(1,-1)$ is an eigenvector with eigenvalue $-1$, and hence a minimum of the function in $\{x^2+y^2\leq 1\}$ will be a point $(x,y)$ on the boundary $\{x^2+y^2=1\} $ and parallel to $(1,-1)$.  We take $X=\{(x,y):~x^2+y^2<1\}$, which is an open subset of $\mathbb{R}^2$, and $r(x,y)=1-\sqrt{x^2+y^2}$ the distance from a point $(x,y)$ to the boundary of $X$. The initial point is $(0.1,0.2)\in X$. 

Newton's method:  after 1 step, arrives at $(0,0)$. 

New Q-Newton's method: after 500 steps, arrives at $(-1.6366953e+149,  1.6366953e+149)$. (Outside the domain.)
 
Random Newton's method: after 22 steps, arrives at $(-2.75384279e-12, -5.50768558e-12)$. 

BFGS: after 2 steps, arrives at $(-144.59463151,  134.26644355)$. (Outside the domain,)

Riemannian Newton's method: after 1 step, arrives at the saddle point $(0,0)$. 

Riemannian New Q-Newton's method: after 50 steps, arrives at $(-0.70710678,  0.70710678)$. (Good.)

Riemannian Backtracking GD: after 50 steps, arrives at $(-0.70707318, 0.70714038 )$. (Good)

{\bf Example 8:} $\min _{||(x,y,z)||\leq 1}f(x,y,z)$, where $f$ is a homogeneous quadratic function in $3$ variables, whose Hessian matrix is 
 \[ \left( \begin{array}{ccc}
-23&-61&40\\
-61&-39.5&155\\
40&155&-50\\
\end{array}\right) \]
This matrix has eigenvalues $0,112.5,-225$. We see that one eigenvector corresponding to the eigenvalue $-255$ is close to $(1/3,2/3,-2/3)$, which belongs to the unit sphere. Hence the minimum of the function in the given domain is about $f(1/3,2/3,-2/3)=-112.5$. We choose $X=\{(x,y,z):~x^2+y^2+z^2<1\}$, which is a bounded open subset of $\mathbb{R}^3$.  We choose $r(x,y,z)=1-\sqrt{x^2+y^2+z^2}$, which is the distance from a point $(x,y,z)$ to the boundary of $X$. The initial point is $(1.188e-05, 2.188e-05, 3.188e-05)$.

Newton's method: after 500 steps, arrives at $(-5.96281624e+13,  8.51830891e+12,$ $-2.12957723e+13)$. (Outside the domain.)

New Q-Newton's method: after 500 steps, arrives at $(5.43407609e+130, -7.74275238e+129,$  $1.93872152e+130)$. (Outside the domain.)
 
Random Newton's method: after 322 steps, arrives at $(1695312.72638973, -242187.53234139,$ $ 605468.83085347)$. (Outside the domain.)

BFGS: Warning "Desired error not necessarily achieved due to precision loss".  After 2 steps, arrives at $(-0.02608408, -0.07921891,$  $0.04140426)$. The function value is $-0.852273$. 

Riemannian Newton's method: after 500 steps, arrives at $(-0.9333333,   0.13333339,$ $-0.33333341)$. The function value is $-1.14e-12$. 

Riemannian New Q-Newton's method: after 500 steps, arrives at $(-0.81909064,  0.29130584,$ $-0.49419776)$. The function value is $-7.06$. 

Riemannian Backtracking GD: after 50 steps, arrives at the point $(-0.33909717, -0.63222429,$  $0.69663875)$. The function value is $-112.14$. (Good.)

{\bf Example 9: } $\min _{||(x,y,z)||\leq 1}-f(x,y,z)$, where $f$ is the function in Example 8. We like to check what happens if the negative eigenvalue is not the dominant of the Hessian matrix. Here, the eigenvector of eigenvalue of the Hessian is $\sim (-0.105263, 0.578947, 0.526316)$. The minimum value in the interested domain is $-112.5$. We define $X$, $r$ and the initial point as in Example 8. 

Newton's method: after 50 steps, arrives at $( 6.81582245, -0.97368892,  2.43422228)$. (Outside of the domain.)

New Q-Newton's method: after 50 steps, arrives at $(1.86654033e+11, -4.53462698e+10,$   $5.26285018e+10)$. (Outside of the domain.) 
 
Random Newton's method: After 50 steps, arrives at $(-2.97077388e-05,  4.24432015e-06,$ $-1.06095969e-05)$. The function value is $-1.22e-17$. 

BFGS: Warning "Desired error not necessarily achieved due to precision loss". After 1 step, arrives at $(-0.3406934,   3.43294092,$  $2.32717428)$. (Outside the domain.)

Riemannian Newton's method: after 50 steps, arrives at $(-0.92624563,  0.13232078,$ $-0.33080201)$. The function value is $4.884981308350689e-15$. 

Riemannian New Q-Newton's method: after 50 steps, arrives at $(-0.94036234,  0.17649288,$ $-0.29013372)$. The function value is $-0.21134058944093814$. 

Riemannian Backtracking GD: after 50 steps, arrives at $(-0.13662457,  0.72666381,  0.6732707)$. The function value is $-56.233306328624224$ (Good). Note that the learning rate is very small $3.33e-14$. 

\subsection{A general method for finding minimum on a bounded ball in Euclidean space - With an application to finding minimum eigenvalue of a symmetric matrix} In the experiments in the previous subsection, we see  that for the optimisation problem as the form $\min _{||x||\leq 1}f(x)$, where $x\in \mathbb{R}$, the Riemannian versions (where the Riemannian manifold in question is $\{x\in \mathbb{R}^m:~||x||<1\}$ with $r(x)=1-||x||$) of Backtracking GD and New Q-Newton's methods can sometimes work well and sometimes work not so well. In this subsection, we show that combining that with Riemnannian optimisation on the sphere $\{x:~||x||=1\}$ can yield improved performance. Hence, we state a general method consisting of 2 steps: 

Step 1: Do Riemannian optimisation on the manifold $\{x\in \mathbb{R}^m:~||x||<1\}$, with $r(x)=1-||x||$ and $R_{x}(v)=x+v$. 

Step 2: Do Riemannian optimisation on the manifold $S^{m-1}=\{x\in \mathbb{R}^m:~||x||=1\}$, with $r(x)=\pi$. (Note that, in experiments, we see that even with putting $r(x)=\infty$, still the performance is very good.) Here, there are two ways to choose $R_x(v)$:

Way 1: $R_x(v)=(x+v)/\sqrt{1+||v||^2}$. 

Way 2: (geodesic) $R_x(v)=\cos (||v||)x+\sin (||v||)v/||v||$. 

Step 3: Compare the performances obtained in Steps 1 and 2, and choose the best one. 

We now describe more details the computation of Step 2 for a quadratic function $f_A(x)=<Ax,x>/2$, where $A$ is a symmetric matrix. In this case, if $\lambda _1(A)$ is the smallest eigenvalue of $A$, then $\min _{x\in S^{m-1}}f(x)=\lambda _1(A)/2$. Hence, this problem is interesting also for numerical linear algebra. 

For $S^{m-1}$, we will use the induced metric from $\mathbb{R}^m$. This implies, in particular that if $v\in T_xS^{m-1}$, then $||v||_{T_xS^{m-1}}=||v||_{R^{m}}$.  The computations for Riemannian gradient and Hessian are also quite nice, the next 2  formulas are taken from \cite[Propositions 3.49, Section 5.5]{boumal}: If $x\in S^{m-1}$ and $v\in T_xS^{m-1}$, then (The RHS of the formulas are interpreted in the usual Euclidean setting)
\begin{eqnarray*}
grad(f_A)(x)&=&Ax-<Ax,x>x,\\
Hess(f_A)(x)[v]&=&Av-<Av,x>x-<Ax,x>v.
\end{eqnarray*} 
 
 As we mentioned above, the Riemannian Hessian is symmetric on $T_xS^{m-1}$. On the other hand, its obvious extension (using the same formula on the RHS) to $T_x\mathbb{R}^m$ may be not symmetric. Since it would be more convenient to do calculations with a {\bf symmetric} extension of $Hess(f_A)(x)$ to the whole $T_x\mathbb{R}^m$ (for example, when we want to decompose into positive and negative eigenvalues as in New Q-Newton's method), we will define explicitly such an extension $B:T_x\mathbb{R}^m\rightarrow T_x\mathbb{R}^m$. The most convenient way is to use, for $v\in T_x\mathbb{R}^m$, its orthogonal projection $v-<v,x>x$ to $T_xS^{m-1}$. Hence, we define $B$ by the formula:
 \begin{eqnarray*}
 B[v]:=Hess(f_A)(x)[v-<v,x>x]. 
 \end{eqnarray*}
 
With the above formulas, we can apply Riemannian Backtracking GD and Riemannian New Q-Newton's method, as well as the Riemannian Newton's method and its random damping version. Since we cannot find codes for Riemannian BFGS in the python library, we do not compare it here. On the other hand, since Riemannian Standard GD is easy to code, we will also compare it. We will choose hyperparameters for Riemannian Backtracking GD as in the previous subsection. For the learning rate for Riemannian Standard GD, we fix it to be $0.001$. We see from the experiments that the convergence here is faster than in the previous subsection. 

{\bf Example 7'.} Consider $\min _{||x||=1}f_A(x)$, where $A$ is the matrix 

 \[ \left( \begin{array}{cc}
2&4\\
4&2\\
\end{array}\right) \]

the same function as in Example 7. The initial point will be $x_0/||x_0||=(0.4472136  ,0.89442719)\in S^1$, where $x_0$ is the initial point in Example 7. 

Riemannian Newton's method: after 10 steps, arrives at $(0.70710678, 0.70710678)$. The function value is $3$. 

Riemannian New Q-Newton's method: after 10 steps, arrives at $(-0.70668054,  0.70753276)$. The function value is $-0.9999985474307601$. (Good.)

Riemannian Random Newton's method: after 10 steps, arrives at $(0.70711097, 0.7071026)$. 

Riemannian Backtracking GD: after 3 steps, arrives at $(-0.70691347,  0.70730003)$. (Good.) Learning rate is $0.11764899999999995$. 

Riemannian Standard GD: after 10 steps, arrives at $(0.42499191, 0.90519715)$. The function value is $2.538805861735874$. 

{\bf Example 8'.} Consider $\min _{||x||=1}f_A(x)$, where $A$ is the matrix: 

 \[ \left( \begin{array}{ccc}
-23&-61&40\\
-61&-39.5&155\\
40&155&-50\\
\end{array}\right) \]

the same function as in Example 8. The initial point will be $x_0/||x_0||=(0.29369586,$ $0.54091459,$ $0.78813333)\in S^2$, where $x_0$ is the initial point in Example 8. 

Riemannian Newton's method: after 10 steps, arrives at $(-0.13333661,  0.73330012,  0.66670254)$. The function value is $56.24999961483963$. It is clear that the convergence seems to be the global {\bf maximum}. 

Riemannian New Q-Newton's method: after 10 steps, arrives at $(-0.3344025,  -0.66691779,  0.66587959)$. The function value is $-112.4997765706635$. (Good.)

Riemannian Random Newton's method: after 10 steps, arrives at $(-0.13360975,  0.73323507,  0.6667194 )$. 

Riemannian Backtracking GD: after 10 steps, arrives at $(-0.33333105, -0.66666699,  0.66666748)$. (Good.)

Riemannian Standard GD: after 10 steps, arrives at $(0.32851449, -0.24217129,  0.91292458)$. The function value is $-40.65387811891485$. 

{\bf Example 9':} Consider $\min _{||x||=1}-f_A(x)$, where $A$ is the matrix in Example 8'. We start from the same initial point $x_0/||x_0||=(0.29369586,$ $0.54091459,$ $0.78813333)\in S^2$.

Riemannian Newton's method: after 10 steps, arrives at $(-0.13333661,  0.73330012,  0.66670254)$. The function value is $-56.24999961483963$. (Good.)

Riemannian New Q-Newton's method: after 10 steps, arrives at $(-0.13333307,  0.7333311,   0.66666918)$. (Good.)

Riemannian Random Newton's method: after 10 steps, arrives at  $(-0.13332444,  0.73331822,  0.66668507)$. (Good.)

Riemannian Backtracking GD: after 10 steps, arrives at at $(-0.13328013,  0.73332264,  0.66668907)$. (Good.)

Riemannian Standard GD: after 10 steps, arrives at $(0.01714908, 0.70172638, 0.7122401)$. The function value is $-54.812314763600995$.

\subsection{Conclusions and interesting open questions}\label{SubsectionConclusions} In this paper, we extended the known theoretical results from \cite{truong-nguyen, truong, truong-etal} for versions of Backtracking GD and New Q-Newton's method to the Riemannian setting. We showed that local information on the geometry and the function are sufficient to guarantee good theoretical properties (convergence and/or avoidance of saddle points)  for the method. This allows flexibility in applications. Examples demonstrate that for a singular cost function or for constrained optimisation on a Euclidean space, it could be beneficial to use  Riemannian optimisation. In particular, we have a general method to deal with constrained optimisation problems of the form $\min _{||x||\leq 1}f(x)$. We did explicit calculations for the case $f$ is a quadratic function, in which case its minimum value is the same as the minimum eigenvalue of the Hessian matrix.  Since the sphere is a {\bf compact} Riemannian manifold, we expect from the theoretical results proven in this paper, that the performance of Riemannian New Q-Newton's method and Riemannian Backtracking GD will be good. 

Here are some research directions which we think are interesting and useful for both theoretical and practical considerations.

{\bf Question 1.} Can we achieve the conclusions of Theorems \ref{Theorem2} and \ref{Theorem3} for the Riemannian Backtracking GD algorithm as well? An answer Yes will be extremely in practice, since Riemannian Backtracking GD is easier to implement than the other modifications of its. We note that in the Euclidean setting, if the function $f$ is $C^2$ and the learning rates $\delta _n$ converging to $0$,  and for GD, avoidance of saddle points has been confirmed in \cite{panageas-piliouras-wang}. (If this learning rate scheme converges, then Armijo's condition is satisfied in the long run.) 

{\bf Question 2.} Can we prove parts 2) and 3) of Theorem \ref{Theorem4} without the assumption that the Strong local retraction is Real analytic-like or without the assumption that the numbers $\delta _0,\ldots ,\delta _m$ in Riemannian New Q-Newton's method are random?

{\bf Question 3.} Can we achieve the conclusions of Theorems \ref{Theorem1}, \ref{Theorem2} and \ref{Theorem3} in the stochastic setting, even in the Euclidean setting only?  We note that in the Euclidean setting, for $C^{1,1}_L$ cost functions having some further restrictions (such as having compact sublevels and the gradient also has compact sublevels), together with the assumption on learning rates similar to those in \cite{robbins-monro},  for Stochastic GD it has been shown in \cite{merti-etal} a property similar  to that $\lim _{n\rightarrow\infty}||x_{n+1}-x_n||=0$ almost surely.  

{\bf Question 4.} Can we achieve the conclusions of Theorem \ref{Theorem4} in the stochastic setting, even in the Euclidean setting only?

{\bf Question 5.} Can we have an efficient implementation of Riemannian New Q-Newton's method in large scale optimisation, even in the Euclidean setting only? To a less difficult level, we ask for an efficient implementation of the Riemannian Backtracking GD and Riemannian Local Backtracking GD in Riemannian manifolds. 

{\bf Question 6.} Can we have a modification of Newton's method which has both: fast performance, convergence guarantee, avoidance of saddle points, and not expensive implementation?  

The results from \cite{panageas-piliouras-wang} and \cite{merti-etal} mentioned in Questions 2 and 3 above, are encouraging first steps toward solving the corresponding question. We recall however from the introduction that the implementation in \cite{merti-etal} for Stochastic GD with diminishing learning rates (learning rates converging to $0$), the experimental results are not that good compared with implementations of Stochastic GD (with constant learning rates) and Backtracking GD. This indicates that in order to have more theoretical justifications of the use of GD in large scale optimisation (such as in Deep Neural Networks), it is better that results are proven under more general and practical assumptions.

\end{document}